\documentclass[12pt]{amsart}
\usepackage{geometry} 
\usepackage{hyperref}

\title{Regularity and Uniqueness of $P$-harmonic Maps  with  Small Range }
\author{Ali  Fardoun and Rachid Regbaoui}
\address{Laboratoire de Math\'ematiques, UMR 6205 CNRS\\ 
Universit\'e de Bretagne Occidentale\\ 
6 Avenue Le Gorgeu, 29238 Brest Cedex 3 \\ 
France}
\date{} 

\begin{document}

\begin{abstract}
We prove the uniqueness of solutions to Dirichlet problem for $p$-harmonic maps with images in a small geodesic ball  of the target manifold. As a consequence, we show  that such maps have  H\"older continuous derivatives. This gives an extension of a result by S. Hildebrandt {\it et al} \cite{sH77}  concerning harmonic maps. \end{abstract} 
\maketitle
\section{Introduction}

Let $(M, g)$ and $(N , h)$ be  compact Riemannian manifolds with dimensions  $m \ge 2$ and $n \ge 2$ respectively. If $p \ge 2$, \  the  $p$-energy of a map $u \in C^1(M, N)$ is defined by 
$$E_p(u) = {1\over p} \int_M  |d u|^p  dg  \  , $$
where $dg$ denotes the volume element of the metric $g$, and 
for each  $x \in M$,   $du(x) : T_xM \to T_{u(x)}N$ is the differential of $u$ at $x$. Here the norm  $|du(x)|$  of the differential of $u$  is given by 
$$|du(x)|^2 = g^{\alpha\beta}(x)h_{ij}(u(x)){\partial u^i\over \partial x_{\alpha}}(x){\partial u^j\over \partial x_{\beta}}(x) ,$$
 where $(g^{\alpha\beta})$ is  the inverse matrix of the metric $g$ in local coordinates $(x_1 , \cdots , x_m)$ on $M$, and  in local coordinates on N$, (u^1, \cdots, u^n )$ are the components of $u$, and $(h_{ij})$  is the matrix of the metric $h$. 
 
\medskip

$P$-harmonic maps are critical points of the functional $E_p$. If $u$ is a smooth $p$-harmonic map, say  \ $C^2$,  then it satisfies the following system of partial differential   equations 
\hypertarget{1.1}{}$$ -\hbox{div}\left(| d u|^{p-2}\nabla u^j\right)   =  | d u|^{p-2}\Gamma^j_{k l}(u) \nabla u^k\nabla u^l \  ,  \     j = 1 , \cdots , n  ,  \eqno (1.1) $$

\noindent  where \     $\Gamma^j_{k l}$ are the Christoffel symbols of the metric $h$. Here  \text{div}  and $\nabla$  denote respectively  the divergence  and the gradient with  respect to the metric $g$ on $M$. In particular, the notation \  $\nabla u^k\nabla u^l $ stands for the Riemannian inner product of $\nabla u^k$ and $\nabla u^l$ with respect to $g$, that's, 
$$\nabla u^k\nabla u^l  =   g^{\alpha \beta}{\partial u^k \over \partial x_{\alpha}}{\partial u^l \over \partial x_{\beta}} \  .  $$

\medskip

The variational  approach to the existence of $p$-harmonic maps  imposes a larger  functions space than $C^1(M, N)$ for the $p$-energy functional.   In order to give a more precise definition  of  critical points of the functional $E_p$, let us recall some  definitions.  Without loss of generality, we may suppose that $(N, h)$ is isometrically embedded in some Euclidean space $\mathbb{R}^{k }$. Then  for any $p \ge 2$, we define the Sobolev space  
$$W^{1 , p} (M , N)  :=  \Bigl\{ \ u  \in  W^{1 , p}(M , \mathbb{R}^k ) \  :  \   u(x) \in N  \  \hbox{a.e}
 \  \Bigr\}  . $$ 
The $p$-energy can be extended to maps $u = (u^1 , \cdots , u^k)\in W^{1 , p} (M , N)$  by 
$$E_p(u) = {1\over p} \int_M  |\nabla u|^p  dg  \  ,  $$
where \ 
$\displaystyle |\nabla u |^p = \left(\sum_{j=1}^{k}|\nabla u^j |^2\right)^{p\over 2},$ and   $\nabla u^j$ is the gradient of $u^j$ with respect to the metric $g$ (and its norm  
$|\nabla u^j |$  is also taken with respect to $g$).  

\bigskip

$P$-harmonic maps are critical points $ u \in W^{1 , p} (M , N)$ of the functional $E_p$ with respect  to  variations of the form
$$\Pi_N (u + t\varphi ) \  , $$ 
where  \  $t \in (-\varepsilon , \varepsilon)$  with $\varepsilon > 0$,  $\varphi \in C^{\infty}_0(M , \mathbb{R}^k ) $  and \  $\Pi_N$ is the nearest point projection from a tubular neigborhood of $N$ in $\mathbb{R}^k$  onto  $N$.  Then we get  the Euler-Lagrange system (in the distributional sense):

\hypertarget{1.2}{}$$-\hbox{div}\left(|\nabla u|^{p-2}\nabla u\right) = |\nabla u |^{p-2} A(u)(\nabla u , \nabla u ) \ ,     \eqno (1.2) $$

\medskip

\noindent where \ $A(u) $ \ is the second fundamental form of $N$ at $u$, and  we are using the notation :
$$A(u)(\nabla u , \nabla u ) :=  g^{\alpha \beta}A(u)\Big( {\partial u \over \partial x_{\alpha}} ,  {\partial u \over  \partial x_{\beta}} \Big)  \ .$$

\medskip

\noindent If in addition  one allows variations of the form  $u\circ \varphi_t$, where $\left(\ \varphi_t \right)_{t \in (-\varepsilon , \varepsilon )}$ is a family of  smooth transformations of $M$ such that $\varphi_0 = I_M $, one gets the so-called stationary $p$-harmonic maps. A map $u  \in   W^{1 , p} (M , N) $ is a minimizing $p$-harmonic map on a domain  $\Omega \subset M$  if, for any $w \in W^{1 , p} (\Omega , N)$ agreeing with $u$ on $\partial\Omega$, we have 
$$\int_{\Omega} |\nabla u|^p \  dg \le  \int_{\Omega} |\nabla w|^p \  dg.  $$
It is not difficult to see that minimizing maps are stationary $p$-harmonic maps. 

\medskip

The difference between system of equations \hyperlink{1.1}{(1.1)} and system \hyperlink{1.2}{(1.2)} is that  in  \hyperlink{1.1}{(1.1)} one needs that  the image of the solution $u$  lie (locally) in a single system  of coordinates  on the target manifold $N$,  which is not necessary concerning  system \hyperlink{1.2}{(1.2)}. However,  one can easily see  that for  $p$-harmonic maps whose images  lie in a single local chart of $N$,   systems 
 \hyperlink{1.1}{(1.1)} and \hyperlink{1.2}{(1.2)} are equivalent.  In particular,  this is the case for continuous $p$-harmonic maps. 
 
\medskip

Solutions of  \hyperlink{1.2}{(1.2)} have to be understood in the weak sense, that's, a map $u \in W^{1, p}(M, N)$ is a solution of  \hyperlink{1.2}{(1.2)} if for any $\varphi \in W^{1, p}_0(M , \mathbb{R}^k) \cap L^{\infty}(M, \mathbb{R}^k)$, we have 
$$\int_M |\nabla u|^{p-2} \nabla u \cdot \nabla \varphi \ dg = \int_M |\nabla u|^{p-2}A(u)(\nabla u , \nabla u) \cdot \varphi \ dg \ ,  $$
where \ the dot denotes the usual inner product in $\mathbb{R}^k$ and the notation \ $\nabla u\cdot \nabla \varphi$ means
$$\nabla u\cdot \nabla \varphi = g^{\alpha \beta}{\partial u \over \partial x_{\alpha}}\cdot {\partial \varphi \over \partial x_{\beta}} \  .$$

System \hyperlink{1.2}{(1.2)} presents two principal difficulties. The first one is that it is quasilinear and degenerate (due to the presence of $|\nabla u|^{p-2}$  in the left hand side). The second difficulty comes from the nonlinear term $|\nabla u |^{p-2}A(u)(\nabla u , \nabla u)$ caused by the geometry of $N$. One has to distinguish the case $p= 2$, which corresponds to harmonic maps, from  the other ones. In fact,  when $p= 2$ we have an elliptic semilinear system,  and  the theory of linear elliptic equations applies  if one has a good control of the right hand side.  The case $p \not= 2$ is more complicated, and one has to take care of the vanishing of the term $ \nabla u $. 

\medskip

 The regularity of harmonic maps has been a very attractive subject of research the last thirty years. The first result is due to S. Hildebrandt {\it et al} \cite{sH77} who proved that harmonic  maps ($p= 2$) whose images lie in a small geodesic ball of the target manifold are smooth. Later, W. J\"ager and H. Kaul \cite{wJ79} proved the uniqueness of solutions to the Dirichlet problem for such maps. As it can be seen by  counterexamples of sphere-valued  maps,  the Dirichlet problem for general harmonic maps may admit more than one solution.   Concerning  partial regularity for harmonic maps, many results were obtained for minimizing harmonic maps, and more generally for stationary harmonic maps.  Such maps are smooth  outside a closed singular set of Hausdorff dimension at most $m-2$, see  \cite{fB93} , \cite{fH91}, \cite{tR08}, \cite{rS82} and references therein.   Concerning $p$-harmonic maps,  the situation is more complicated since the system of equations (1.2) is quasilinear and degenerate. To our knowledge, the only known uniqueness result  for Dirichlet problems of $p$-harmonic maps is due to the first author \cite{aF05}, and it concerns maps with values in the Euclidean sphere $\mathbb{S}^n$.  Concerning the partial regularity, it was shown by  R. Hardt and F. Lin \cite{rH87}, and S. Luckhauss \cite{sL88}, that minimizing $p$-harmonic maps are \  $C^{1, \alpha}$, $ 0< \alpha < 1$,   outside a closed singular set of Hausdorff dimension at most $m - [p] -1 $. In  \cite{mF89},  M. Fuchs   generalised the result of S. Hildebrandt {\it et al }\cite{sH77} above to $p$-harmonic maps but with the additional assumption that the map is  stationary. (see also  \cite{aG05},  \cite{pS03} for related results  on removable singularities of $p$-harmonic maps). 
 
  \medskip
 
 In this paper we prove  the uniqueness of solutions to the Dirichlet problem for $p$-harmonic maps with values in small geodesic balls of the target manifold. As a consequence, we prove that such maps  are  \ $C^{1,  \alpha}$ for some $0 < \alpha < 1$ and are minimizing the $p$-energy. 
 
 \medskip
 
\newtheorem{theo}{Theorem}[section]
 \hypertarget{t2}{} \begin{theo} Let $M$ be a compact Riemannian manifold with smooth boundary $\partial M$, $N$ a compact  Riemannian manifold  and $p \ge 2$. There is a constant $\varepsilon_0 =  \varepsilon_0 (N , p) > 0$  
depending only on $N$  and $p$  such that if  $u , v  \in W^{1 , p}(M , N)$  are $p$-harmonic maps satisfying 
$$u(M) \subset  B(P_0 , \varepsilon_0) \    ,  \   v(M) \subset  B(P_0 , \varepsilon_0)     \   \hbox{and }   \   u = v  \ \hbox{on } \ \partial M,  $$
where \  $ B(P_0 , \varepsilon_0 )$ is a geodesic ball of $N$ of radius $\varepsilon_0$ centered at some point $P_0 \in N,$     then $ u= v$ on $M$. 
 \end{theo}

\bigskip

Theorem 1.1 allows us  to prove the following regularity result :

  \hypertarget{t1}{}\begin{theo} Let  $M$ and $N$ be compact Riemannian manifolds, and let $p \ge 2$. There is a constant $\varepsilon_1 =  \varepsilon_1 ( N , p) > 0$ \
depending only on $N$  and $p$  such that if $u \in W^{1 , p}(M , N)$ is a $p$-harmonic map satisfying 
$$u(\Omega) \subset B(P_0 , \varepsilon_1),$$ 
where \ $\Omega$ is an open set of $M$ and $ B(P_0 , \varepsilon_1 )$ is a geodesic ball of $N$ of radius $\varepsilon_1$ centered at some point $P_0 \in N$,  then 
$u \in C^{1 , \alpha}(\Omega, N)$ for some $0< \alpha < 1$.  Moreover, $u$ is minimizing the $p$-energy in $\Omega$  among maps agreeing with $u$ on $\partial \Omega$ and having their values in $B(P_0, \varepsilon_1)$.   \end{theo} 

\bigskip

Theorem 1.1  and Theorem 1.2 can be seen respectively as  extensions of the results  of   W. J\"ager and H. Kaul \cite{wJ79} and S. Hildebrandt {\it et al} \cite{sH77} above  to $p$-harmonic maps.   Due to the lack of ellipticity, the regularity $ C^{1, \alpha}$  in  Theorem 1.2  is the best one could expect in general for $p$-harmonic maps. 

\medskip

 It would be interesting to compute the optimal value of the constant $\varepsilon_0$  in Theorem \hyperlink{t1}{1.1},  and $\varepsilon_1$ in Theorem \hyperlink{t2}{1.2}. When $p=2$, S. Hildebrandt {\it  et al} \cite{sH77}  and W. J\"ager and H. Kaul \cite{wJ79} computed respectively   upper bounds of  the constant  $\varepsilon_0$ and $\varepsilon_1$,  they found 
$$\varepsilon _0 = \varepsilon_1< \inf(i_N,  {\pi \over 2 \kappa}) \ , $$ where $\kappa > 0$  is an upper bound of the sectional curvature of $N$, and $i_N$ is the injectivity radius of $N$.    These bounds are optimal as it can be seen by considering sphere-valued harmonic maps (see \cite{sH84}). 
For stationary $p$-harmonic maps,  M.Fuchs\cite{mF89}  obtained $\varepsilon_1< \inf(i_N,  {\pi \over 4 \kappa})$.  In the case that  $N = \mathbb{S}^n$, the first author found  in \cite{aF05}   the optimal bounds  \ $\varepsilon _0= \varepsilon_1 <  {\pi \over 2}$ \ . 

\medskip

The paper is organized as follows. In section \hyperlink{s2}{2} we  prove Theorem \hyperlink{t2}{1.1}. In section \hyperlink{s3}{3} we show the  existence of  a minimizing  $p$-harmonic map with small range, and then  we  combine this result with Theorem \hyperlink{t2}{1.1}  to prove Theorem \hyperlink{t1}{1.2}.

\bigskip

\hypertarget{s2}{}\section{Uniqueness of $p$-harmonic maps with small range }

   \bigskip

This section is devoted to the proof of Theorem \hyperlink{t2}{1.1}, which  needs  some preliminary results. In what follows,  $M$ is a compact  Riemannian manifold eventually with boundary, and $N$ is a compact Riemannian manifold without boundary, which is isometrically embedded in some Euclidean space $\mathbb{R}^k$. For $P_0 \in N$ and $r >0$, we denote by $B(P_0, r) $ the Euclidean  open ball in $\mathbb{R}^k$ of radius $r$,  centered at $P_0$. One of the  principal ingredients in the proof of Theorem \hyperlink{t2}{1.1} is the following stability inequality. 

\bigskip

\hypertarget{p1}{}\newtheorem{prop}{Proposition}[section]  \begin{prop} There exists a constant $C_N $ depending only on $N$ such that  if   $u \in W^{1, p}(M, N)$  is a  $p$-harmonic map satisfying 
$u(M) \subset B(P_0 , r )$, for some $P_0 \in N$ and  $0 < r < C_N$, then  we have, for any \ $\varphi \in W^{1, p}_0(M, \mathbb{R}^k)$,  
$$\int_M |\nabla u|^p |\varphi|^2 dg \le  16 r^2   \int_M |\nabla u |^{p-2}|\nabla \varphi |^2 dg  \ .$$ \end{prop}

\bigskip

\begin{proof}  By a limiting argument, it suffices to prove the proposition for $\varphi \in C_0^{\infty}(M, \mathbb{R}^k)$ since this space is dense in $W^{1, p}_0(M, \mathbb{R}^k)$.  
Let then $\varphi  \in C_0^{\infty}(M, \mathbb{R}^k)$ and take  $|\varphi|^2(u-P_0)$ as testing function in   the system \hyperlink{1.2}{(1.2)}, we get 
\hypertarget{2.1}{}$$\int_M |\nabla u|^p |\varphi|^2 dg   \  \le  \   2  \int_M |\nabla u|^{p-1}  |u-P_0| |\nabla \varphi| |\varphi| \ dg     $$
$$ + \  \int_M |\nabla u|^{p-2} \left|A(u)(\nabla u , \nabla u )\right| |u-P_0| |\varphi|^2  \ dg . \eqno (2.1) $$

\medskip

 Since $N$ is compact and $|u-P_0| < r$, we have by the bilinearity of the second fundamental form,    $\left|A(u)(\nabla u , \nabla u)\right| |u-P_0 | \le C_0 r |\nabla u|^2$, where $C_0$ is a constant depending on $N$. Then the last term in the right hand side of \hyperlink{2.1}{(2.1)}  satisfies
\hypertarget{2.2}{} $$ \int_M |\nabla u|^{p-2} \left|A(u)(\nabla u , \nabla u )\right| |u-P_0| |\varphi|^2  dg   \le C_0 r \int_M |\nabla u|^p |\varphi|^2 dg \ .   \eqno (2.2)$$
 
 \noindent On the other hand, by the Cauchy-Schwarz inequality we have 
\hypertarget{2.3}{}$$  2  \int_M |\nabla u|^{p-1}|u-P_0| |\nabla \varphi| |\varphi| \ dg \le $$
$$2r \left(\int_M |\nabla u|^p |\varphi|^2 dg \right)^{1\over 2} \left(
 \int_M |\nabla u |^{p-2}|\nabla \varphi |^2 dg\right)^{1\over 2} . \eqno (2.3)$$
 
 \medskip
 
 It follows from \hyperlink{2.1}{(2.1)}, \hyperlink{2.2}{(2.2)}  and \hyperlink{2.3}{(2.3)}  that 
 
$$ (1-C_0 r) \int_M |\nabla u|^p |\varphi|^2 dg \le 2r \left(\int_M |\nabla u|^p |\varphi|^2 dg \right)^{1\over 2} \left(
 \int_M |\nabla u |^{p-2}|\nabla \varphi |^2 dg\right)^{1\over 2} ,$$
 which gives, if we suppose $C_0 r  < {1 \over 2} $, 
 $$\left( \int_M |\nabla u|^p |\varphi|^2 dg\right)^{1\over 2}  \le  4 r \left( \int_M |\nabla u |^{p-2}|\nabla \varphi |^2 dg\right)^{1\over 2}. $$
 The proposition is then proved.
 \end{proof}
 
 \bigskip
 
 For the proof  Theorem \hyperlink{t2}{1.1} we need also two lemmas.  The first one   concerns some  inequalities on Euclidean spaces  that  we will prove for the convenience of the reader. 
 
 \bigskip
 
 \hypertarget{l1}{}\newtheorem{lem}{Lemma}[section]  \begin{lem} Let  \ $V$ be a real vector space endowed with an inner product. For $X , Y  \in V,$  we denote by  \ $X\cdot Y$  \ the inner product of $X$ and $Y$, and  by  $|X| = \sqrt{X\cdot X}$ the associated norm of $X$. Then for any $q \ge 0$,  and for any $ X , Y \in V$, we have 
 \hypertarget{2.4}{}$$\left( |X|^qX - |Y|^qY\right)\cdot(X - Y) \ge {1\over 2}\left(|X|^q + |Y|^q \right)|X-Y|^2 \   \eqno (2.4) $$
 and 
 \hypertarget{2.5}{} $$  \left||X|^qX - |Y|^qY \right| \le (q+1)\left( |X|^q + |Y|^q \right)  |X - Y| \ .\eqno (2.5)$$
 
  \end{lem}
 
 \bigskip
 
\begin{proof}   Let us first prove \hyperlink{2.4}{(2.4)}. We may suppose, without loss of generality,   that $|X| = 1$.  Let $V_{X, Y}$ be  a $2$-dimensional vector subspace of $V$ containing the  vectors $X$ and $Y$.  Set $e_1 = X$ and choose a vector $e_2 \in  V_{X, Y}$ such that \ $\{ \  e_1 , e_2 \  \}$  is an orthonormal basis of $V_{X, Y}$. Taking coordinates with respect to this basis,  we have  
$$X = (1 , 0)   \   \    \hbox{and}   \    \     Y = \left( r\cos\theta   ,   r \sin\theta \right)  \   , \   \hbox{ with} \   r = |Y|  , \   \theta \in [0 , 2 \pi ] . $$
Then inequality  \hyperlink{2.4}{(2.4)} can be written 
$$1 + r^{q+2} - r(1+r^q)\cos\theta \ge {1\over 2} (1+r^q)(1+r^2 - 2r\cos\theta ) \ ,  $$

\noindent which is equivalent to 
$$1 + r^{q+2} - r^{q} -  r^{2}  \ge 0 . $$
But the last inequality is always true since \  $1 + r^{q+2} - r^{q} -  r^{2}   = (1-r^2)(1-r^q)$ \ and $q \ge 0$.  This proves \hyperlink{2.4}{(2.4)}.

\medskip

\noindent Now,  to  prove \hyperlink{2.5}{(2.5)} we  set  $F(Z) = |Z|^qZ$. Then by the mean value theorem, we have 
$$ \left||X|^qX - |Y|^qY\right|  = \left|F(X) - F(Y)\right|  \le   \sup_{0 \le t \le 1}\|dF( tX + (1-t)Y )\|  |X-Y| , $$
where $dF(Z) : V \to V$ is the differential of $F$ at $Z $, and where we denote by $\| L\| $ the norm of any bounded linear map  $L : V \to V$.   A straightforward  computation gives,  for any $Z \in V$, 
$$ \|dF(Z)\| \le (q + 1)|Z|^q  ,$$
which implies 
$$ \left| |X|^qX - |Y|^qY\right|  \le (q+1)  \sup_{0 \le t \le 1}\left| tX + (1-t)Y\right|^q |X-Y| $$
$$ \le (q+1) \sup_{0 \le t \le 1}\left( t |X| + (1-t)|Y| \right)^q|X-Y| $$
$$ \le (q+1) \max\left(|X|^q , |Y|^q \right)|X-Y| $$
$$\le  (q+1) \left(|X|^q + |Y|^q \right)|X-Y|  .$$
This proves (2.5). 
\end{proof}

 \bigskip
 
 In the following lemma, we prove an inequality satisfied by the second fundamental form of $N$. For any $y \in N$, we denote by $T_yN \subset \mathbb{R}^k$ the tangent space of $N$ at $y$, and $A(y) : T_yN \times T_yN \to T_yN^{\perp}$ the second fondamental form of $N$ at $y$. 
 
 \begin{lem} There is a constant $C$ depending only on $N$ such that, for any $y, z \in N$, and for any $Y \in T_yN,  \ Z \in T_zN$,  we have 
 \hypertarget{2.6}{}$$\left|A(y)(Y , Y) - A(z)(Z , Z) \right| \le $$
 $$C\left( |Y|^2 + |Z|^2 \right)|y-z| + C \left( |Y| + |Z| \right) |Y - Z| , \eqno (2.6)$$
 where $| \ .\ | $ denotes the Euclidean norm in $\mathbb{R}^k$. \end{lem}
 
 \begin{proof} Since $N$ is smooth and compact, there exist $\delta> 0$, and a finite number of points $y_1 , \cdots , y_{K} \in N$ such that 
 $$ N \subset \bigcup_{\nu=1}^{K}B(y_\nu , {\delta \over 2} ) , $$
 and  for each $\nu= 1, \cdots , K$,   a smooth  orthonormal frame  $\{ e_i(y) \}_{1 \le i\le n}$  on  $N \cap B(y_\nu, \delta)$   of the tangent space of  $N$   at   $y \in  N\cap B(y_\nu , \delta )$.   Let  $y, z  \in  N$ and  $Y \in T_yN,  \  Z \in T_zN$. If $|y-z| \ge {\delta\over 2}$, then we have  by the bilinearity of the second fondamental form :
 $$ \left|A(y)(Y , Y) - A(z)(Z , Z) \right| \le C\left( |Y|^2 + |Z|^2 \right) $$
 $$ \hspace{1cm} \le 2C\delta^{-1}\left( |Y|^2 + |Z|^2 \right)|y-z| , $$
 where $C$ is a positive constant depending only on $N$. This  proves \hyperlink{2.6}{(2.6)} in this case. Now suppose  that $|y-z| < {\delta\over 2}$. Then $y, z \in B(y_{\nu_0} , \delta )$ for some $1\le \nu_0 \le K$, and consider a smooth orthonormal frame    $\displaystyle \{ e_i \}_{1 \le i\le n}$  on $N \cap B(y_{\nu_0}, \delta)$. For each \ $i, j = 1, \cdots, n$,  consider the map  $F_{i j} : B(y_{\nu_0} , \delta )\cap N \to \mathbb{R}^k$ defined by 
 $$F_{ij}(x) = A(x)\bigl(e_i(x), e_j(x)\bigr) \ , \    x\in  B(y_{\nu_0} , \delta )\cap N . $$
 Since $F_{ij}$ is a smooth map, then we have by the mean value theorem 
 $$|F_{ij}(y) -F_{ij}(z) | \le C|y-z| $$
   for a constant   $C$ depending only on $N$, that's 
 \hypertarget{2.7}{}$$\left|A(y)(e_i(y) , e_j(y)) - A(z)(e_i(z) , e_j(z)) \right| \le C |y-z| . \eqno (2.7)$$
 We have also by smoothness of the map $F_{ij}$, \  $|F_{ij}(y)| \le  C$ \ , \  $|F_{ij}(z)|\le C$, that's 
 
 \hypertarget{2.8}{}$$ \left| A(y)\bigl(e_i(y), e_j(y)\bigr)\right| \le C \     \text{and} \    \left| A(z)\bigl(e_i(z), e_j(z)\bigr)\right| \le C .  \eqno (2. 8 ) $$
 
 \noindent If we take two tangent vectors $Y \in T_yN$ and $Z \in T_zN$, then $ \displaystyle Y = \sum_{i= 1}^n \alpha_i e_i(y) $ \ and \  $\displaystyle Z = \sum_{i= 1}^n \beta_i e_i(z) $ for some $(\alpha_1, \cdots , \alpha_n) \in \mathbb{R}^n$ and \  $(\beta_1, \cdots , \beta_n) \in \mathbb{R}^n$. Set $\alpha =  (\alpha_1, \cdots , \alpha_n) $ and $\beta = (\beta_1, \cdots , \beta_n)$, then 
 \hypertarget{2.9}{}$$ |Y| =|\alpha| \  , \  |Z| = |\beta| . \eqno (2.9)$$
 
 \medskip
 
\noindent  Now, we have 

   $$ A(y)(Y , Y) - A(z)(Z , Z)  = \sum_{i , j =1}^{n}\Bigl( \alpha_i\alpha_j A(y)(e_i(y) , e_j(y)) -  \beta_i\beta_j A(z)(e_i(z) , e_j(z)) \Bigr). $$
  Then 
 $$ | A(y)(Y , Y) - A(z)(Z , Z)| \le  \sum_{i , j =1}^{n}|\alpha_i\alpha_j|\left|A(y)(e_i(y) , e_j(y)) - A(z)(e_i(z) , e_j(z)) \right| $$
 $$ + \sum_{i , j =1}^{n}|\alpha_i\alpha_j - \beta_i\beta_j |\left|A(z)(e_i(z) , e_j(z)) \right| $$
 which gives by using \hyperlink{2.7}{(2.7)} and \hyperlink{2.8}{(2.8)} 
 \hypertarget{2.10}{} $$  | A(y)(Y , Y) - A(z)(Z , Z)| \le  C \sum_{i , j =1}^{n}|\alpha_i\alpha_j| |y-z|  + C   \sum_{i , j =1}^{n}|\alpha_i\alpha_j - \beta_i\beta_j | $$
 $$ \le  C \sum_{i , j =1}^{n}|\alpha_i\alpha_j| |y-z|  + C  \sum_{i , j =1}^{n}|\alpha_i| |\alpha_j - \beta_j | + C \sum_{i, j = 1}^n |\beta_j| |\alpha_i - \beta_i| $$
 $$\le C |\alpha|^2|y-z| + C (|\alpha | + |\beta| ) |\alpha - \beta |  .  \eqno (2.10)$$
 
 \medskip
 
 \noindent On the other hand, we have 
 \hypertarget{2.11}{} $$|Y-Z| = \left| \sum_{i=1 }^{n} \alpha_ie_i(y) - \beta_ie_i(z) \right|  \ge \left| \sum_{i=1}^{n} (\alpha_i - \beta_i)e_i(y) \right| - \left| \sum_{i=1 }^{n} \beta_i (e_i(y) - e_i(z) ) \right|$$
 $$ \ge |\alpha - \beta | - \sum_{i=1 }^{n} |\beta_i | | e_i(y) - e_i(z)| .    \eqno (2.11) $$
  But by using again  the mean value theorem, we have, for any $i = 1, \cdots , n$, 
  $$|e_i(y) - e_i(z) | \le C |y-z| $$
  where $C$ depends only on $N$.  Hence it follows from \hyperlink{2.11}{(2.11)}  that 
  $$|Y-Z|   \ge |\alpha - \beta| - C |\beta| |y-z| ,$$ 
  that's 
  $$|\alpha - \beta | \le |Y-Z| + C|\beta| |y-z| .$$
  
  \medskip
  
  \noindent Combining this inequality with \hyperlink{2.10}{(2.10)} gives 
  
  \medskip
  
  $$  | A(y)(Y , Y) - A(z)(Z , Z)| \le  C(|\alpha |^2 + |\beta|^2 ) |y-z| + C (|\alpha| + |\beta| )|Y-Z|, $$
  
  \medskip
  
 \noindent  and since by \hyperlink{2.9}{(2.9)} we have $|\alpha | = |Y|$, $|\beta| = |Z|$, we finally  obtain 
 \medskip
  $$ | A(y)(Y , Y) - A(z)(Z , Z)| \le  C(|Y |^2 + |Z|^2 ) |y-z| + C (|Y| + |Z| )|Y-Z| .$$
  This proves the lemma. 
  
  \bigskip
  
  We are now in position to prove Theorem \hyperlink{t2}{1.1} . 
   
 \end{proof}

 \begin{proof}[Proof of Theorem \hyperlink{t1}{1.1}] Let $u, v \in W^{1, p}(M, N)$ \  be $p$-harmonic maps such that $u= v$ on $\partial M$. Since $u$ and $v$ are solutions of  system \hyperlink{1.2}{(1.2)}, then  we have for any   $\varphi \in W^{1, p}_0(M, \mathbb{R}^k)\cap L^{\infty}(M, \mathbb{R}^k)$, 
 
 $$ \int_M  |\nabla u|^{p-2}\nabla u \cdot \nabla\varphi  \  dg  =   \int_M |\nabla u|^{p-2} A(u)(\nabla u, \nabla u) )\cdot \varphi \ dg $$
 and
 $$ \int_M  |\nabla v|^{p-2}\nabla v \cdot \nabla\varphi  \  dg  =   \int_M |\nabla v|^{p-2} A(v)(\nabla v, \nabla v) )\cdot \varphi \ dg\ .$$
 
 \medskip
 
 \noindent Taking the difference of the last two equations,  and choosing  $\varphi = u - v  \in W^{1, p}_0(M, \mathbb{R}^k)\cap L^{\infty}(M, \mathbb{R}^k)$ (since $u-v = 0$ on $\partial M$), we get 
 
 \hypertarget{2.12}{}$$\int_M  \left(|\nabla u|^{p-2}\nabla u  - |\nabla v|^{p-2}\nabla v \right)\cdot \nabla (u-v)  \  dg  = $$
$$ \int_M \left(|\nabla u|^{p-2} A(u)(\nabla u, \nabla u) - |\nabla v|^{p-2}A(v)(\nabla v , \nabla v)\right)\cdot (u-v)  \  dg \  . \eqno (2.12) $$

\medskip 

 By choosing a local orthonormal frame on $M$,  we can identify \ $\nabla u$ and $\nabla v$ with vectors in $\mathbb{R}^{\nu }$ (endowed with its  usual inner product), where $\nu = mn$. Then inequality \hyperlink{2.4}{(2.4)} of   Lemma \hyperlink{l1}{2.1}  applied to $X= \nabla u$ and $Y = \nabla v$,  with \ $V= \mathbb{R}^{\nu}$  and $q = p-2$, \  gives 
$$ \left(|\nabla u|^{p-2}\nabla u  - |\nabla v|^{p-2}\nabla v \right)\cdot \nabla (u- v)  \ge {1\over 2} \left(|\nabla u|^{p-2} + |\nabla v|^{p-2}\right)|\nabla (u - v) |^2 .$$
Then it follows from \hyperlink{2.12}{(2.12)} that
\hypertarget{2.13}{}$$ \int_M \left(|\nabla u|^{p-2} + |\nabla v|^{p-2}\right)|\nabla ( u - v) |^2  \ dg \le $$ 
$$2 \int_M \left| |\nabla u|^{p-2} A(u)(\nabla u, \nabla u) - |\nabla v|^{p-2}A(v)(\nabla v , \nabla v)\right| |u-v|  \  dg  \ .  \eqno (2.13) $$

\medskip

\noindent On the other hand, taking  again a  local orthonormal frame \ $\{ e_j \}_{1\le j\le m}$  on $M$, we  recall that 
$$ A(u)(\nabla u, \nabla u) = \sum_{j= 1}^mA(u)(\nabla_{e_j}u ,  \nabla_{e_j}u )  $$ 
and 
$$A(v)(\nabla v, \nabla v)  = \sum_{j=1}^mA(v)(\nabla_{e_j}v ,  \nabla_{e_j}v ) .$$
Then applying Lemma 2.2 with $y = u , z = v $, and 
  $Y= |\nabla u|^{p-2 \over 2}\nabla_{e_i} u $\ ,  $ Z = |\nabla v|^{p-2 \over 2}\nabla_{e_i} v$,   we have 

 \hypertarget{2.14}{}$$\left| |\nabla u|^{p-2} A(u)(\nabla u, \nabla u) - |\nabla v|^{p-2}A(v)(\nabla v , \nabla v) \right| \le  $$
$$\sum_{i=1}^m\Bigl| |\nabla u|^{p-2} A(u)(\nabla_{e_i} u, \nabla_{e_i} u) - |\nabla v|^{p-2}A(v)(\nabla_{e_i} v , \nabla_{e_i} v) \Bigr| \le$$
$$C\sum_{i=1}^m\left(  |\nabla u|^{p-2} |\nabla_{e_i} u|^2 +  |\nabla v|^{p-2} |\nabla_{e_i} v|^2\right)|u-v|  $$
$$ + \  C \sum_{i=1}^m\left(  |\nabla u|^{p-2\over 2} |\nabla_{e_i} u| +   |\nabla v|^{p-2\over 2} |\nabla_{e_i} v|\right)\left| |\nabla u|^{p-2\over 2} \nabla_{e_i} u -
|\nabla v|^{p-2\over 2} \nabla_{e_i} v\right| . \eqno (2.14)$$
But we have 
$$ \sum_{i=1}^m\left(  |\nabla u|^{p-2} |\nabla_{e_i} u|^2 +  |\nabla v|^{p-2} |\nabla_{e_i} v|^2\right)|u-v|  = \left(|\nabla u|^p + |\nabla v|^p\right)|u-v| $$
and by Cauchy-Schwarz inequality we have 
$$\sum_{i=1}^m\left(  |\nabla u|^{p-2\over 2} |\nabla_{e_i} u| +   |\nabla v|^{p-2\over 2} |\nabla_{e_i} v|\right)\left| |\nabla u|^{p-2\over 2} \nabla_{e_i} u -
|\nabla v|^{p-2\over 2} \nabla_{e_i} v\right|  \le $$
$$\left( \sum_{i=1}^m |\nabla u|^{p-2} |\nabla_{e_i} u|^2\right)^{1/2} \times
  \left( \sum_{i=1}^m\left| |\nabla u|^{p-2\over 2} \nabla_{e_i} u -|\nabla v|^{p-2\over 2} \nabla_{e_i} v  \right|^2\right)^{1/2}  $$
$$ +  \  \left( \sum_{i=1}^m |\nabla v|^{p-2} |\nabla_{e_i} v|^2\right)^{1/2} \times
  \left( \sum_{i=1}^m\left| |\nabla u|^{p-2\over 2} \nabla_{e_i} u -|\nabla v|^{p-2\over 2} \nabla_{e_i} v  \right|^2\right)^{1/2}  $$
  
  \medskip
  
$$= \Bigl( |\nabla u|^{p\over 2} + |\nabla v|^{p\over 2} \Bigr) \left| |\nabla u|^{p-2\over 2} \nabla u -|\nabla v|^{p-2\over 2} \nabla  v \right| .$$

\bigskip

\noindent Hence it follows from \hyperlink{2.14}{(2.14)} that 
 \hypertarget{2.15}{}$$\left| |\nabla u|^{p-2} A(u)(\nabla u, \nabla u) - |\nabla v|^{p-2}A(v)(\nabla v , \nabla v) \right| \le  $$
$$C \Bigl( |\nabla u|^{p} + |\nabla v|^p\Bigr)|u-v|  + C \Bigl( |\nabla u|^{p \over 2} + |\nabla v|^{p\over 2} \Bigr) 
\left| |\nabla u|^{p-2\over 2} \nabla u -|\nabla v|^{p-2\over 2} \nabla  v \right| . \eqno (2.15) $$

\medskip

\noindent By identifying the vectors $\nabla u$ and $\nabla v$ with vectors in $\mathbb{R}^{mn}$, and applying inequality \hyperlink{2.5}{(2.5)} in Lemma \hyperlink{l1}{2.1} with $V= \mathbb{R}^{mn}$, 
$X= \nabla u$ ,  $Y = \nabla v$,  and $q = {p-2 \over 2}$, we have 
$$\left| |\nabla u|^{p-2\over 2} \nabla u -|\nabla v|^{p-2\over 2} \nabla  v \right|  \le {p\over 2}  \Bigl( |\nabla u|^{p - 2 \over 2} + |\nabla v|^{p -2 \over 2}\Bigr)
|\nabla u - \nabla v | .$$ 

\noindent Combining this inequality with \hyperlink{2.15}{(2.15)}, we get 

 \hypertarget{2.16}{}$$\left| |\nabla u|^{p-2} A(u)(\nabla u, \nabla u) - |\nabla v|^{p-2}A(v)(\nabla v , \nabla v) \right| \le$$
$$  C_p \left( |\nabla u|^{p} + |\nabla v|^{p}\right)|u-v| + C_p\left( |\nabla u|^{p-1} + |\nabla v|^{p-1}\right)|\nabla(u-v)|  ,  \eqno (2.16) $$

\medskip

\noindent where $C_p$ is a positive constant depending only on $N$ and $p$.  

\medskip

Putting  \hyperlink{2.16}{(2.16)} in \hyperlink{2.13}{(2.13)}, we obtain 

$$ \int_M \left(|\nabla u|^{p-2} + |\nabla v|^{p-2}\right)|\nabla ( u - v) |^2  \ dg \le    C_p\int_M \left( |\nabla u|^{p} + |\nabla v|^{p}\right)|u-v|^2 \ dg ,$$
$$ + \  C_p\int_M\left( |\nabla u|^{p-1} + |\nabla v|^{p-1}\right)|\nabla(u-v)| |u-v|  \ dg  $$

\medskip

\noindent which by the Cauchy-Schwarz inequality gives 

\bigskip

 \hypertarget{2.17}{}$$ \int_M \left(|\nabla u|^{p-2} + |\nabla v|^{p-2}\right)|\nabla ( u - v) |^2  \ dg \le C_p \int_M \left( |\nabla u|^{p} + |\nabla v|^{p}\right)|u-v|^2 \ dg  $$
$$ +  \  C_p \left(  \int_M\left(|\nabla u|^{p-2} + |\nabla v|^{p-2}\right)|\nabla (u-v)|^2  dg\right)^{1/2}$$
$$\times \left( \int_M\left(|\nabla u|^{p}+ |\nabla v|^{p} \right) |u-v|^2 dg\right)^{1/2}  . \eqno (2.17) $$

\medskip

Now,  if the constant $\varepsilon_0$ in Theorem \hyperlink{t2}{1.1} is sufficiently small, then the geodesic ball $B(P_0 , \varepsilon_1)$ in $N$  is included in the Euclidean ball $B(P_0, r)$ in $\mathbb{R}^k$ with $r$  as in Proposition \hyperlink{p1}{2.1}.  Applying Proposition \hyperlink{p1}{2.1} to both $u$ and $v$ with $\varphi = u-v  \in W^{1, p}_0(M, \mathbb{R}^k)$,  and taking the sum,  we have 
$$\int_M \left(|\nabla u|^{p} + |\nabla v|^{p}\right)|u-v|^2 \ dg \le $$
$$ 16 r^2 \int_M \left(|\nabla u|^{p-2} + |\nabla v|^{p-2}\right)|\nabla ( u - v) |^2  \ dg .$$ 
Combining this inequality with \hyperlink{2.17}{(2.17}), we finally obtain 
$$ \int_M \left(|\nabla u|^{p-2} + |\nabla v|^{p-2}\right)|\nabla ( u - v) |^2  \ dg \le $$ 
$$  C_p \left( 16r^2 + 4r  \right)   \int_M \left(|\nabla u|^{p-2} + |\nabla v|^{p-2}\right)|\nabla ( u - v) |^2  \ dg \ .$$
Thus if $r $ is small  enough,   we obtain $\nabla u- \nabla v = 0$, and then $u= v$ since $u= v$ on $\partial M$. This proves Theorem \hyperlink{t2}{1.1}. 
\end{proof}

\bigskip

\hypertarget{s3}{} \section{Energy minimizing maps with small range }

\bigskip

Our goal in this section is the proof of Theorem \hyperlink{t1}{1.2}. To this end, we need  the following  proposition on the existence of a minimizing  $p$-harmonic map,  agreeing with $u$ on the boundary of the domain,  and having its image in the same geodesic ball as $u$.  In what follows, we set 
\hypertarget{3.1}{}$$r_N = \inf(i_N,  {\pi \over 2 \kappa}) ,  \eqno (3.1)$$
where \ $i_N$ is the injectivity radius of the manifold $N$ and $\kappa \ge 0$ is an upper bound of the sectional curvature of $N$. 

\bigskip

\hypertarget{p2}{}\newtheorem{pro}{Proposition}[section]  \begin{pro}  Let \  $  0 < r  <  r_N$,  and let $u \in W^{1, p}(D , N)$ satisfying \  $ u(D) \subset B(P_0 , r)$,  where  $D \subset M$ is an open set with smooth boundary $\partial D$,  and $B(P_0 , r)$ is a  geodesic ball in $N$ of radius $r$,  centered at some point $P_0 \in N$. Then there exists a  $p$-harmonic map  \  $ v \in  W^{1, p}(D , N)$ such that  \  $ v= u $ on $ \partial D$,  $v(D) \subset B(P_0 , r) $ \ and satisfying 
$$ \int_D |\nabla v |^p dg  \le \int_D  |\nabla w |^p dg  $$
for all \   $ w \in   W^{1, p}(D , N) $ \  with  \  $ w(D) \subset B(P_0 , r) $   and    $w = u$  on   $ \partial D$ .     \end{pro}

\bigskip

\begin{proof}  The proof  is an adaptation of a similar result by   S. Hildebrandt {\it et al} \cite{sH77} in the harmonic case ($p=2$). The existence of $v$ relies on classical variational arguments, so the  main difficulty is to prove that $v$ is $p$-harmonic. Fix $r_1 $ such that  $r < r_1 < r_N$. Since the functional  $E_p$ is bounded from below, there is a minimizing sequence  $(v_l) $ of the $p$-energy   in $ W^{1, p}(D , N) $ \  such that   \  $ v_l(D) \subset \overline{B(P_0 , r_1)} $   and    $v_l = u$  on   $ \partial D$ .  Up to a subsequence, we may suppose that $(v_l)$ converges weakly  in $ W^{1, p}(D , \mathbb{R}^k) $ and strongly in $L^p(D, \mathbb{R}^k)$   to  some $v \in  W^{1, p}(D , N)$,  with  \  $ v(D) \subset \overline{B(P_0 , r_1)} $   and    $v = u$  on   $ \partial D$. Moreover, by lower semi-continuity of the $p$-energy functional, we have 
$$ \int_D |\nabla v |^p dg \le \liminf_{l\to \infty}  \int_D |\nabla v_l |^p dg .$$

\noindent This proves that $v$ is minimizing the $p$-energy  on $D$ among maps having their values in $\overline{B(P_0 , r_1)}$ and agreeing with $u$ on $\partial D$.  It remains to show  that  $ v(D) \subset B(P_0 , r)$ and   that $v$ is $p$-harmonic in $D$. We follow the arguments of   \cite{sH77}, and we shall prove that $v$ satisfies the system \hyperlink{1.1}{(1.1)} (in the weak sense ) in normal coordinates around $P_0$. 
 Let   $v = (v^1 , \cdots , v^n )$  be the expression of $v$ in  such normal coordinates. Then  we have to prove that, for any  $ \varphi = (\varphi^1 , \cdots, \varphi^n) \in W^{1, p}_0(D , \mathbb{R}^n)\cap L^{\infty}(D, \mathbb{R}^n)$, 
\hypertarget{3.2}{}$$\int_D |dv|^{p-2}\nabla v^i\nabla\varphi^i \ dg = \int_D |dv|^{p-2}\Gamma^{j}_{kl}(v)\varphi^j\nabla v^k\nabla v^l \ dg .  \eqno (3. 2)   $$
 First let us prove that \ $v(D) \subset  B(P_0 , r)$, that's, \ $|v| < r $  (recall that we are working in  normal coordinates around $P_0$). 
Let $0  \le \eta \in W^{1 , p}_0(D , \mathbb{R})\cap L^{\infty}(D, \mathbb{R})$ and consider the function $v_t = \exp_{P_0}\left((1 - t \eta )v\right)$,  with $ 0 \le t  \le \|\eta\|_{\infty}^{-1}$.   Then  $v_t \in    W^{1, p}(D , N)$, $v_t = u$ on  $\partial D$,   $v_t(D) \subset \overline{B(P_0, r_1)}$, and $v_0 = v$.  Hence, since $v$ is minimizing, we have $E_p(v) \le E_p(v_t)$ for any $ 0 \le t \le  \|\eta\|_{\infty}^{-1}$.   If we take the derivative  with respect to $t$ at $t = 0$, we get 
$${d \over dt}E_p(v_t)\Big|_{t=0} \ \ge 0 \  , $$
which gives (after some computations) 
$$\int_D |dv|^{p-2}\nabla v^i\nabla(\eta v^i) \ dg - \int_D \eta |dv|^{p-2}\Gamma^{j}_{kl}(v)v^j\nabla v^k\nabla v^l  \ dg  \  \le \  0 \  .  $$

Now  if we choose \ $ \eta = \hbox{max}(|v|^2 - r^2 , 0)$, then it is easy to see that $ \eta  \in W^{1 , p}_0(D , \mathbb{R})\cap L^{\infty}(D, \mathbb{R})$ and that $\nabla\eta = 2 \nabla v\cdot v $ \   if  $  |v| \ge r $ and 
                        $\nabla \eta =  0   $ \ if  $ |v| \le r$.   Replacing $\eta$ in the last inequality, we obtain 
$${1\over 2} \int_D |dv|^{p-2}|\nabla \eta|^2 dg+  \int_D\eta |dv|^{p-2}\left( |\nabla v|^2 -  \Gamma^{j}_{kl}(v)v^j\nabla v^k\nabla v^l \right) dg \le 0.$$
Since $|v| \le r_1$  and $r_1 < r_N$, we have, according to a result proved in \cite{sH84} ( inequality (6.11)  ), 
$$ |\nabla v|^2 -  \Gamma^{j}_{kl}(v)v^j \nabla v^k\nabla v^l  \ \ge 0 .$$
Thus we obtain
$$\int_D |dv|^{p-2}|\nabla \eta|^2  dg \le 0$$
which gives   \  $\nabla \eta = 0$, that's,  $ |v| < r$.  It remains then to prove \hyperlink{3.2}{(3.2)}. Let $ \varphi \in W^{1, p}_0(D , \mathbb{R}^n)\cap L^{\infty}(D, \mathbb{R}^n)$, and consider the map $v_t = \exp_{P_0}\left(v + t\varphi\right) $. Then since $|v| <  r < r_1$, we have  \ $ v_t(D) \subset B(P_0 , r_1)$, for  any  \  $ t \in (-\delta , \delta )$, \ with  \ $\delta = (r_1 -r)\|\varphi\|_{\infty}^{-1}> 0$, and $v_t = u$ on $\partial D$. 
Since $v$ is minimizing the $p$-energy functional among $W^{1, p}$ maps having their values in the ball $\overline{B(P_0 , r_1)}$ and agreeing with $u$ on $\partial D$, we have \ $E_p(v) \le E_p(v_t)$ for any $t \in (-\delta , \delta )$. By taking the derivative with respect to $t$ of $E_p(v_t)$ at $t = 0$, we obtain 
$${d \over dt}E_p(v_t)\Big|_{t=0} = 0$$
which gives \hyperlink{3.2}{(3.2)}.  This achieves the proof of Proposition \hyperlink{p2}{3.1}. 
\end{proof}

\bigskip

\begin{proof}[Proof of Theorem \hyperlink{t1}{1.2}]   Let $\varepsilon_0$ as in Theorem \hyperlink{t2}{1.1} and  let $\varepsilon_1 > 0$ such that 
$$\varepsilon_1 <   \inf\left(i_N , {\pi \over 4 \kappa} ,  \varepsilon_0\right)$$
where \ $i_N$ is the injectivity radius of the manifold $N$ and $\kappa \ge 0$ is an upper bound of the sectional curvature of $N$. 
 Let $u$ as in Theorem \hyperlink{t1}{1.2}   such that $u(\Omega) \subset B(P_0 , \varepsilon_1)$, and  let $D \subset \Omega $ a domain with smooth boundary $\partial D$. By Proposition \hyperlink{p2}{3.1}, with $r= \varepsilon_1$, there is a $p$-harmonic  map $v \in W^{1 , p}(D , N)$  such that  $ v= u $ on $ \partial D$,   $v(D) \subset B(P_0 , \varepsilon_1) $ \ and satisfying 
$$ \int_D |\nabla v |^p dg  \le \int_D  |\nabla w |^p dg $$
for any  \   $ w \in   W^{1, p}(D , N) $ \  with  \  $ w(D) \subset B(P_0 , \varepsilon_1) $   and    $w = u$  on   $ \partial D$. 
Since $v$ is minimizing the $p$-energy, then it is easy to see that $v$ is a stationary $p$-harmonic map.  It follows from the regularity result of M. Fuchs \cite{mF89} for  stationary $p$-harmonic maps with image in geodesic balls of radius \ $ r <  \inf(i_N , {\pi \over 4 \kappa} )$,  that $v$ is in $C^{1, \alpha}(D, N)$ for some  $0< \alpha < 1$. But by  Theorem \hyperlink{t2}{1.1}    we have \ $u = v$ since we are supposing $\varepsilon_1 < \inf(i_N , {\pi \over 4 \kappa} ,  \varepsilon_0)  \le  \varepsilon_0 $. This achieves the proof of Theorem \hyperlink{t1}{1.2}. 

\end{proof}


\end{document}